%% file: facredpsdcompl.tex
\documentclass[11pt]{article}
\usepackage{cite}
\usepackage{lineno}
\usepackage{refcheck}    % adds info on where labels are used
\usepackage{mathrsfs} % Ting Kei added
\usepackage{ifthen}
\usepackage{adjustbox}
\usepackage{siunitx}
\usepackage{multirow}
\usepackage{exscale}
\usepackage{tabularx}
\usepackage[section]{algorithm} 
\usepackage{algorithmic}     % was not available on PC?
\usepackage{latexsym}
\usepackage{amsmath,amssymb,amstext,amsthm} % Lots of math symbols and environments
\usepackage{graphicx,color,epsfig}
\usepackage{graphicx}
\usepackage{fullpage}        % was not available on PC?
\usepackage{float}
\usepackage{verbatim}
\usepackage{datetime}
\usepackage[bookmarks,pagebackref,
    pdfpagelabels=true, % Adds page number as label in Acrobat's page count
    ]{hyperref}
\usepackage{seceqn}
\usepackage{makeidx}
\makeindex

\def\R{\mathbb{R}}

\def\Rp{\R^p}
\def\Rt{\R^t}

\def\N{\mathbb{N}}

\def\eqref#1{{\normalfont(\ref{#1})}}

\def\FR{\mbox{\boldmath$FR$}\,\,}
\def\FRp{\mbox{\boldmath$FR$}}

\def\NNMp{\mbox{\boldmath$NNM$}}
\def\LRMC{\mbox{\boldmath$LRMC$}\,\,}
\def\LRMCp{\mbox{\boldmath$LRMC$}}
\def\SDP{\mbox{\boldmath$SDP$}\,\,}
\def\SDPp{\mbox{\boldmath$SDP$}}

\def\eqref#1{{\normalfont(\ref{#1})}}
\newtheorem{defi}{Definition}[section]

\newtheorem{prop}{Proposition}[section]

\newtheorem{cor}{Corollary}[section]

\newtheorem{lemma}{Lemma}[section]

\newcommand{\textdef}[1]{\textit{#1}\index{#1}}

\newcommand{\Ss}{{\mathcal S} }

\newcommand{\PP}{{\mathcal P} }

\newcommand{\norm}[1]{\left\| #1 \right\|}

\newcommand{\Sn}{{\mathcal S^n}}
\newcommand{\Smn}{{\mathcal S^{m+n}}}

\newcommand{\MM}{{\mathcal M}}

\newcommand{\A}{{\mathcal A}}

\newcommand{\Zcal}{{\mathcal Z}}

\newcommand{\bbm}{\begin{bmatrix}}
\newcommand{\ebm}{\end{bmatrix}}
\newcommand{\bem}{\begin{pmatrix}}
\newcommand{\eem}{\end{pmatrix}}
\newcommand{\beq}{\begin{linenomath*} \begin{equation}}
\newcommand{\beqs}{\begin{linenomath*} \begin{equation*}}
\newcommand{\eeq}{\end{equation} \end{linenomath*}}
\newcommand{\eeqs}{\end{equation*} \end{linenomath*}}
\newcommand{\beqr}{\begin{linenomath*} \begin{eqnarray}}
\newcommand{\beqrs}{\begin{linenomath*} \begin{eqnarray*}}
\newcommand{\eeqr}{\end{eqnarray} \end{linenomath*}}
\newcommand{\eeqrs}{\end{eqnarray*} \end{linenomath*}}
\newcommand{\bet}{\begin{table}}

\DeclareMathOperator{\trace}{{trace}}
\DeclareMathOperator{\Trace}{{trace}}

\DeclareMathOperator{\size}{{size}}
\DeclareMathOperator{\eig}{{eig}}

\DeclareMathOperator{\rank}{{rank}}

\newcommand{\nc}{\newcommand}
\nc{\arrow}{{\rm arrow\,}}
\nc{\Arrow}{{\rm Arrow\,}}
\nc{\BoDiag}{{\rm B^0Diag\,}}
\nc{\bodiag}{{\rm b^0diag\,}}

\nc{\Mm}{{\mathcal M}^{m} }
\nc{\Mmn}{{\mathcal M}^{mn} }
\nc{\Mpq}{{\mathcal M}^{pq} }
\nc{\Mnr}{{\mathcal M}_{nr} }
\nc{\Mnmr}{{\mathcal M}_{(n-1)r} }
\nc{\kwqqp}{Q{$^2$}P\,}
\nc{\kwqqps}{Q{$^2$}Ps}

\nc{\notinaho}{(X,S)\in \overline{AHO}(\A)}
\nc{\inaho}{(X,S)\in AHO(\A)}

\newcommand{\bea}{\begin{eqnarray}}%
\newcommand{\eea}{\end{eqnarray}}%
\newcommand{\beas}{\begin{eqnarray*}}%
\newcommand{\eeas}{\end{eqnarray*}}%
\newcommand{\Rmn}{\R^{m \times n}}%
\newcommand{\Range}{\mathcal{R}}%
{}
\newcommand{\Hnp}[1][]{\,\mathbb{H}_+^{\ifthenelse{\equal{#1}{}}{n}{#1}}}
\newcommand{\Hn}[1][]{\,\mathbb{H}^{\ifthenelse{\equal{#1}{}}{n}{#1}}}
\newcommand{\Dn}[1][]{\,\mathbb{D}^{\ifthenelse{\equal{#1}{}}{n}{#1}}}

\begin{document}
%\begin{multicols}{2}
%\setlength{\multicolsep}{30pt}

\bibliographystyle{plain}

\title{
	Low-Rank  Matrix Completion using Nuclear Norm\\ 
	with Facial Reduction\footnote{Presented as Part of tutorial at
		DIMACS Workshop on Distance Geometry: Theory and
	Applications, July 26-29, 2016, \cite{Wolktalkdimacs:16}}
%\footnote{Department of Combinatorics and Optimization,
%          Waterloo, Ontario N2L 3G1, Canada,
%          Research Report CORR 2010-01.}
%\footnote{Research supported by Natural Sciences Engineering Research
%		   Council Canada and a grant from AFOSR.}
}

\author{
	\href{http://shimenghuang.com/} {Shimeng Huang}
	\thanks{
	Department of Combinatorics and Optimization, Faculty of Mathematics, University of Waterloo, Waterloo, Ontario, Canada N2L 3G1}
\and
	\href{http://www.math.uwaterloo.ca/~hwolkowi/} {Henry Wolkowicz}%
    \thanks{Department of Combinatorics and Optimization, Faculty of Mathematics, University of Waterloo, Waterloo, Ontario, Canada N2L 3G1; Research supported by The Natural Sciences and Engineering Research Council of Canada and by AFOSR;}
%\and
%	\href{https://www.linkedin.com/in/xinghang-ye-03526497?} {Xinghang Ye}
%	\thanks{Department of Computational Mathematics, Faculty of Mathematics, University of Waterloo, Waterloo, Ontario, Canada N2L 3G1}
}

\date{\today, \currenttime}
          \maketitle
%\begin{center}
%          University of Waterloo\\
%          Department of Combinatorics and Optimization\\
%          Waterloo, Ontario N2L 3G1, Canada\\
%          Research Report CORR 2009-04
%\end{center}

%{\bf Key Words:}  

%\noindent {\bf AMS Subject Classification:}

\begin{abstract}
Minimization of the nuclear norm is often used as a surrogate, convex
relaxation, for finding the minimum rank completion (recovery)
of a partial matrix.
%This parallels the compressed sensing framework.
The minimum nuclear norm problem can be solved as a trace minimization
semidefinite programming problem (\SDP). The \SDP and its dual are regular
in the sense that they both satisfy strict feasibility. Interior point
algorithms are the current methods of choice for these problems. This
means that it is difficult to solve large scale problems and difficult
to get high accuracy solutions.

In this paper we take advantage of the structure at optimality for the
minimum nuclear norm problem. We show that even though strict
feasibility holds, the facial reduction framework can be successfully applied to
obtain a proper face that contains the optimal set, and thus can
dramatically reduce the size of the final nuclear norm problem while
guaranteeing a low-rank solution. We include numerical tests for 
both exact and noisy cases. In all cases we assume that knowledge
of a \emph{target rank} is available.
\end{abstract}

%\linenumbers


\tableofcontents
\listoftables
\listofalgorithms
%\listoffigures

\section{Introduction}
We consider the intractable
\emph{low-rank matrix completion problem} (\LRMCp), i.e.,~the
\index{low-rank matrix completion problem, \LRMCp}
\index{\LRMCp, low-rank matrix completion problem}
problem of finding the missing elements of a given (partial) matrix so
that the completion has low-rank. This problem can be relaxed using the
nuclear norm that can be then solved using a 
\emph{semidefinite programming (\SDPp)}
\index{semidefinite programming (\SDPp)}
\index{\SDPp, semidefinite programming}
model. Though the resulting \SDP and its dual satisfy strict feasibilty,
we show that it is implicitly highly degenerate and amenable to 
\emph{facial reduction} (\FRp). 
\index{facial reduction (\FRp)}
\index{\FRp, facial reduction}
This is done by taking advantage of the special structure 
\emph{at the optimum} and by using the
\textdef{exposing vectors} approach,~see~\cite{ChDrWo:14} 
The exposing
vector approach is particularly amenable to the noisy case. Moreover,
the result from facial reduction is a significant reduction in the size
of the variables and a decrease in the rank of the solution.  
If the data is exact, then \FR results in redundant constraints that we
remove before solving for the low-rank solution.
While if the data is contaminated with noise, \FR yields an
overdetermined semidefinite least squares problem.
We \emph{flip} this problem to minimize the nuclear norm 
using a Pareto frontier approach. 
Instead of removing constraints from the overdetermined problem, 
we exploit the notion of \textdef{sketch matrix} to reduce the size of the 
overdetermined problem. The sketch matrix approach is studied in 
e.g.,~\cite{PolanciWainwright:15}. 

The problem of low-rank matrix completion has many applications to
model reduction, sensor network localization, pattern recognition and
machine learning. 
This problem is further related to real applications in data science, 
for instance, the collaborative filtering (the well known Netflix problem) 
and multi-tasking learning.  See e.g.,~the recent work 
in \cite{MR3440180,MR3436116} and the references therein.
Of particular interest is the case where the data is contaminated with noise. 
This falls into the area of \emph{compressed sensing} or
\emph{compressive sampling}.
An extensive collection of papers, books, codes is available at 
the: Compressive Sensing Resources,
\href{http://dsp.rice.edu/cs}{http://dsp.rice.edu/cs}.

The convex relaxation of minimizing the rank using
the nuclear norm, the sum of the
singular values, is studied in e.g.,~\cite{Rechtparrilofazel,Fazel:02}. 
The solutions can be found directly by subgradient methods or by
using semidefinite programming \SDP with interior point methods or
low-rank methods, again see~\cite{Rechtparrilofazel}. Many
other methods have been developed, e.g.,~\cite{MekaJainDhillon:09}.
The two main approaches for rank minimization,
convex relaxations and spectral methods, are discussed in
\cite{MR3417786,MR2723472} along with a new algebraic combinatorial approach.
A related analysis from a different viewpoint using
rigidity in graphs is provided in \cite{MR2595541}.
%%% \cite{MR2723472}. candes/tao power of convex relaxation 

%\subsection{Outline}
We continue in Section \ref{sect:backgr} with the basic notions for
\LRMC using the nuclear norm and with the graph
framework that we employ. We continue in Section \ref{sect:nuclnorm} with
the details on how to exploit \emph{facial reduction} \FRp, for the
\SDP model to minimize the nuclear norm problem.
\index{\FRp, facial reduction}
Section \ref{sect:dimred} presents the details for finding the low-rank
solution after the \FR has been completed.
We present the numerical results in Section \ref{sect:numers} and a
comparison with results in \cite{MR3440180}.
Concluding remarks are included in Section \ref{sect:concl}.

\section{Background on \LRMCp, Nuclear Norm Minimization, \SDP}
\label{sect:backgr}
We now consider our problem within
the known framework on relaxing the low-rank matrix completion
problem using the nuclear norm minimization and then using 
\SDP to solve the relaxation. For the known results
we follow and include much of the known development in the literature 
e.g.,~\cite[Prop. 2.1]{Rechtparrilofazel}.
In this section we also include several useful tools and
a graph theoretic framework that allows us to exploit \FR at the
optimum.

\subsection{Models}
Suppose that we are given a \textdef{partial $m\times n$ real matrix $Z\in
\Rmn$} which has precise data. The \emph{low-rank matrix completion problem} 
\LRMCp,
\index{low-rank matrix completion problem, \LRMCp}
\index{\LRMCp, low-rank matrix completion problem}
can be modeled as follows:
\begin{equation}
\label{basicsetting}
(\LRMCp) \quad \begin{array}{cl}
	\min & \rank(M)\\
	\text{s.t.}& \PP_{\hat{E}}(M) = b,
\end{array}
\end{equation} 
where \textdef{$\hat E$} is the set of 
indices containing the known entries of $Z$, 
$\PP_{\hat{E}}(\cdot): \Rmn \rightarrow \R^{\hat E}$ is the projection onto
the corresponding entries in $\hat{E}$, and
$b = \PP_{\hat{E}}(Z)$ is the vector 
of known entries formed from $Z$.
However the rank function is not a convex function and the
\LRMC is computationally intractable.

To set up the problem as a convex optimization problem, we can relax 
the rank minimization using \textdef{nuclear norm minimization}:
\begin{equation}
\label{basicnuclear}
(\NNMp) \qquad
\begin{array}{cl}
	\min & \|M \|_* \\
	\text{s.t.}& \PP_{\hat{E}}(M) = b,
\end{array}
\end{equation}
where the nuclear norm $\|\cdot\|_*$ is the sum of the singular
values, i.e., $\|M \|_* = \sum_i \sigma_i(M)$.

Moreover,  we consider the primal-dual pair of problems for
the \textdef{nuclear norm} minimization
problem:
\begin{equation}
	\label{eq:pdpairnuclnorm}
	\qquad
	\begin{array}{lr}
	\begin{array}{cl}
		\min_M & \|M\|_* \\
		\text{s.t.} & \A(M) = b
	\end{array}
	\qquad
	\qquad
	\qquad
	\begin{array}{cl}
		\max_z & \langle b,z\rangle  \\
		\text{s.t.} & \|\A^* (z) \| \leq 1,
	\end{array}
	\end{array}
\end{equation}
where  $\A : \Rmn \to \Rt$ is a linear mapping, 
$\A^*$ is the \textdef{adjoint of $\A$}, and
$\|\cdot\|$ is the operator norm of a matrix, i.e.,~the largest
singular value. The matrix norms  $\|\cdot\|_*$ and $\|\cdot\|$
are a dual pair of matrix norms akin to the vector $\ell_1,\ell_\infty$
norms on the vector of singular values. Without loss of generality, we further
assume that $\A$ is \emph{surjective}.\footnote{Note that $\A$
corresponding to sampling is surjective as we can consider
$\A(M)_{ij \in \hat E} = \trace (E_{ij} M)$, where $E_{ij}$
is the $ij$-unit matrix.} In general, the linear equality constraint is
an underdetermined linear system. In our case, we restrict to the case
that $\A=\PP_{\hat E}$.

\begin{prop}
Suppose that there exists $\hat M$ with $\A(\hat M) = b$.
The pair of programs in \eqref{eq:pdpairnuclnorm} are a convex
primal-dual pair and they
satisfy both primal and dual strong duality, i.e.,~the 
optimal values are equal and both values are attained. 
\end{prop}
\begin{proof}
This is shown in \cite[Prop. 2.1]{Rechtparrilofazel}. That
primal and dual strong duality holds can be seen from the fact that the
generalized Slater condition trivially holds for both programs using
$M=\hat M, z=0$.
\end{proof}

\begin{cor}
The optimal sets for the primal-dual pair in \eqref{eq:pdpairnuclnorm}
are nonempty, convex, compact sets.
\end{cor}
\begin{proof}
This follows since both problems are regular, i.e.,~since $\A$ is
surjective the primal satisfies the \textdef{Mangasarian-Fromovitz
constraint qualification}; while $z=0$ shows that the dual satisfies
strict feasibility. It is well known that this constraint qualification
is equivalent to the dual problem having a nonempty, convex, compact
optimal set, e.g.,~\cite{Fia:83}.
\end{proof}

The following Proposition shows that, we can embed the problem into an \SDP and solve it efficiently. 
\begin{prop}
The pair in \eqref{eq:pdpairnuclnorm} are equivalent to the following
\SDP primal-dual pair:
\begin{equation}
	\label{eq:pdpairnuclnormSDP}
	\begin{array}{lr}
	\begin{array}{cl}
		\min & \frac 12  \trace \left(W_1+W_2 \right) \\
		\text{s.t.} & Y = \begin{bmatrix}
		W_1  & M \cr
		M^T & W_2
		              \end{bmatrix} \succeq 0 \\
	    & \A (M) = b
	\end{array}
	\qquad
	\qquad
	\qquad
	\begin{array}{cl}
		\max_z & \langle b,z\rangle  \\
		\text{s.t.} & \begin{bmatrix}
		I_m  & \A^* (z) \cr
		\A^* (z)^T & I_n
		              \end{bmatrix} \succeq 0.
	\end{array}
	\end{array}
\end{equation}
\qed
\end{prop}

This means that after ignoring the $\frac 12$ we can further 
transform the sampling problem as:
\begin{equation}
\label{sdpnuclear}
\begin{array}{cl}
	\min & \|Y\|_* = \trace(Y) \\
	\text{s.t.}& \PP_{\bar{E}}(Y) = b\\
			  & Y \succeq 0,
\end{array}
\end{equation}
where \textdef{$\bar {E}$} is the set of indices in $Y$ that
correspond to $\hat E$, the known entries of the upper right block of
$\Zcal = \begin{bmatrix} 
0 & Z \\ Z^T & 0 \end{bmatrix}  \in \Smn$.
Here $Y \succeq 0 $ denotes the L\"owner partial order that
$Y$ is \textdef{positive semidefinite, $Y \in \Ss_+^{m+n}$}.
\index{$Y\succeq 0$, positive semidefinite}

When the data are contaminated with noise, we reformulate the strict equality constraint
by allowing the observed entries in the output matrix to be perturbed within 
a tolerance $\delta$ for the norm, where $\delta$ is normally a known 
noise level of the data, i.e.,
\begin{equation}
\label{sdpnuclearinexact}
\begin{array}{cl}
	\min & \|Y\|_* = \trace(Y) \\
	\text{s.t.}& \|\PP_{\bar{E}}(Y) - b\| \leq \delta\\
			  & Y \succeq 0.
\end{array}
\end{equation}

We emphasize that since there is no 
constraint on the diagonal blocks of $Y$, we can always obtain a 
positive definite feasible solution in this exact case
by setting the diagonal elements of $Y$ to be large enough. Therefore
strict feasibility, the \textdef{Slater constraint qualification}, 
always holds.

\subsection{Graph Representation of the Problem}
For our needs, we furthermore associate $Z$ with the 
\textdef{weighted undirected graph, $G=(V,E,W)$}, 
\index{$G=(V,E,w)$, weighted undirected graph}
with \textdef{node set}
$V=\{1,\ldots,m,m+1,\ldots,m+n\}$, \textdef{edge set $E$}, that satisfies
\index{$E$, edge set}
\index{$V$, node set}
\[
	\big\{ \{ij\in V\times V: i< j\leq m\}
	\cup \{ij\in V\times V: m+1\leq i < j\leq m+n\}\big\}
\subseteq E\subseteq \{ij\in V\times V: i< j\},
\]
and \textdef{weights} for all $ij\in E$
\[
W_{ij} = \begin{cases} Z_{i(j-m)}, &\forall ij\in \bar{E}\\ 0, &\forall
ij\in E \backslash \bar{E}. \end{cases}
\]
Note that as above,
$\bar{E}$ is the set of edges excluding the trivial ones, that is, 
\[
	\bar{E} = E \backslash \bigg\{ \{ij\in V\times V: i\leq j\leq m\}
\cup \{ij\in V\times V: m+1\leq i\leq j\leq m+n\}\bigg\}.
\]

We can now construct the \textdef{adjacency matrix, $A$}, for the
graph $G$ as follows
\index{$A$, adjacency matrix}
\begin{equation}
	\label{eq:Aadj}
	A_{ij}= \left\{ 
		\begin{array}{cl}
			1   & \text{if  } ij \in E \text{ or } ji \in E \cr
			0   & \text{otherwise}.
		\end{array}
		\right.
\end{equation}
Recall that a \textdef{clique} in the graph $G$ is a complete subgraph in $G$.
We have the trivial cliques
$C=\{i_1,\ldots, i_k\}\subset \{1,\ldots, m\}$ and
$C=\{j_1,\ldots, j_k\}\subset \{m+1,\ldots, m+n\}$, which are not of
interest to our algorithm.  The nontrivial cliques of interest
correspond to (possibly after row and column permutations) a 
full (specified) submatrix $X$ in $Z$. 
The cliques of interest are $C=\{i_1,\ldots, i_k\}$ with cardinalities
\begin{equation}
	\label{eq:cardspq}
	|C\cap \{1,\ldots, m\}|=p \neq 0, \quad
	|C\cap \{m+1,\ldots, m+n\}|=q \neq 0.
\end{equation}
This means that we have found
\begin{equation}
	\label{eq:Xspecif}
	X\equiv \{Z_{i(j-m)}: ij \in C\}, \quad \text{specified (fully known)
	$p\times q$ rectangular matrix}.
\end{equation}
These non-trivial cliques are at the center of our considerations.

%For $Y\in \Smn$, and let $P_{\bar{E}}(Y)\in \R^{\bar{E}}$ denote the projection of $Y$
%onto the elements with indices in $\bar{E}$, the set of vertices excluding the trivial ones $[\{ij\in V\times V: i\leq j\leq m\}
%	\cup \{ij\in V\times V: m+1\leq i\leq j\leq n\}]$.
\index{$\Sn$, space of real symmetric matrices}
\index{space of real symmetric matrices, $\Sn$}

\index{$\norm{M}_*$}

\section{Facial Reduction, Cliques, Exposing Vectors}
\label{sect:nuclnorm}
In this section we look at the details of solving the \SDP formulation of
the nuclear norm relaxation for \LRMCp. In
particular we show how to
exploit cliques in the graph $G$ and the \emph{special structure at the
optimum}. We note again that though strict feasibility holds for the \SDP
formulation, we can take advantage of facial reduction and efficiently
obtain low-rank solutions.

\subsection{Structure at Optimum}

The results in Section \ref{sect:backgr} can now be used to prove the following special
structure at the optimum. This structure is essential in our \FR scheme.
\begin{cor}
Let $M^*$ be optimal for the primal in \eqref{eq:pdpairnuclnormSDP} with
$\rank(M^*)=r_M$. Then there exist variables $W_1,W_2,z$ to complete the 
primal-dual pair for \eqref{eq:pdpairnuclnormSDP} such that
the compact spectral decomposition of the corresponding
optimal $Y$ in \eqref{eq:pdpairnuclnormSDP} can be written as
\begin{equation}
	\label{eq:YDUV}
0\preceq Y= \begin{bmatrix}
		W_1  & M^* \cr
	(M^*)^T & W_2
		              \end{bmatrix}
= 
\begin{bmatrix} U   \cr V \end{bmatrix}
	D
\begin{bmatrix} U   \cr V \end{bmatrix}^T
= 
\begin{bmatrix}  UDU^T & UDV^T     \cr 
		       VDU^T & VDV^T \end{bmatrix}, \quad
D \in \Ss^{r_Z}_{++},\, \rank Y=:\textdef{$r_Y$}=r_M.
\end{equation}
We get
\begin{equation}
	\label{eq:Winorm}
	W_1=UDU^T, \quad W_2=VDV^T, \quad
	 M^*=UDV^T, \qquad \|M^*\|_*=\frac 12\trace(Y)=\frac 12\trace(D).
 \end{equation}
\end{cor}
\begin{proof}
Let $M^*=U_M\Sigma_M V_M^T$ be the compact SVD with $\Sigma_M \in \Ss^{r_M}_{++}$
on the diagonal. Let 
\[
D=2\Sigma_M, \quad U=\frac 1{\sqrt{2}}U_M,\, V=\frac 1{\sqrt{2}}V_M,
\qquad
Y=\begin{bmatrix} U   \cr V \end{bmatrix}
	D
\begin{bmatrix} U   \cr V \end{bmatrix}^T.
\]
Then the matrix $\begin{bmatrix} U   \cr V \end{bmatrix}$ has
orthonormal columns and $\trace Y=2\trace(\Sigma_M)=2\|M\|_*$.
Therefore \eqref{eq:Winorm} holds. Since $Y$ is
now primal optimal and Slater's condition holds for the primal
problem, there must exist $z$ optimal for the dual.
\end{proof}

%\subsubsection{Exploiting Special Structure at Optimum for \FRp}
Now suppose that there is a \textdef{specified submatrix, $X \in \R^{p \times q}$}, of 
\index{$X \in \R^{p \times q}$, specified submatrix}
 $Z\in \Rmn, \rank(X) =\textdef{$r_X$}$. Without loss of generality,
 after row and column permutations if needed, we can assume that
\[
	Z=
\begin{bmatrix}  Z_1  & Z_2     \cr 
		       X & Z_3 \end{bmatrix},
\]
and we have a full rank factorization $X=\bar P\bar Q^T$ obtained using
the compact SVD 
\[
X=\bar P\bar Q^T=U_XD_X V_X^T, \, D_Z \in \Ss^{r_X}_{++},\quad
\bar P=U_X D_X^{1/2},\,
\bar Q=V_X D_X^{1/2}.
\]
Note that a desirable $X$ that corresponds to a clique in $G$ is given by
\[
	\textdef{$C_X=\{i,\ldots,m,m+1,\ldots,m+k\}$}, \qquad
r <\max\{p, q\},
\]
where we denote the \textdef{target rank, $r$}. We can
exploit the information using these cliques to obtain exposing vectors
of the \textdef{optimal face}, i.e.,~the smallest face of $\Ss^{m+n}_+$
that contains the set of optimal solutions.
\index{$r$, target rank}

By abuse of notation, we can rewrite the optimality form in
\eqref{eq:YDUV} as
\begin{equation}
	\label{eq:Ypartit}
0\preceq Y 
= 
\begin{bmatrix} U \cr P \cr Q  \cr V \end{bmatrix}
	D
\begin{bmatrix} U \cr P \cr Q  \cr V \end{bmatrix}^T
= 
\left[
\begin{array}{c|cc|c}  
	UDU^T & UDP^T& UDQ^T& UDV^T     \cr 
	\hline
	PDU^T & PDP^T& PDQ^T& PDV^T     \cr 
	QDU^T & QDP^T& QDQ^T& QDV^T     \cr 
	\hline
	VDU^T & VDP^T& VDQ^T& VDV^T     \cr 
		       \end{array}
\right].
\end{equation}
We see that $X=PDQ^T = \bar P\bar Q^T$. 
Since $X$ is \emph{big enough}, we conclude that generically $r_X=r_Y=r$,
see Lemma \ref{statlemma} below, and that the ranges satisfy
\begin{equation}
	\label{eq:PQbar}
\Range(X) = \Range(P) = \Range(\bar{P}), \qquad
\Range(X^T) = \Range(Q) = \Range(\bar{Q}).
\end{equation}
This is the key for facial reduction as we can use an \textdef{exposing vector}
formed from $\bar P$ and/or $\bar Q$. 
\begin{lemma}[Basic \FRp] Let $r<\min\{p,q\}$ and
let  \eqref{eq:Ypartit}, \eqref{eq:PQbar} hold with $X=PDQ^T = \bar P\bar Q^T$, 
found using the full rank factorization.  Let $Y$ be an optimal solution 
of the primal problem in \eqref{eq:pdpairnuclnormSDP}.
Define $(\bar U,\bar V)=\FRp(\bar P,\bar Q)$ by
\begin{equation}
	\label{eq:qpuv}
	\FRp(\bar P,\bar Q): \quad
	\bar P \bar P^T + \bar U \bar U^T \succ 0, \quad \bar P^T \bar U=0,  
	\qquad
	\bar Q \bar Q^T + \bar V \bar V^T \succ 0, \quad \bar Q^T \bar V=0.
\end{equation}
By abuse of notation, suppose that both matrices
$\bar U \leftarrow \bar U \bar U^T,\bar V \leftarrow \bar V \bar V^T$ are
filled out with zeros above and below
so their size is that of $Y$ and let $W=
\bar U+ \bar V$. Then $\bar U, \bar V, W$ are all exposing vectors for
the optimal face, i.e.,~for $W$ we have
 $W\succeq 0, WY=0$. Moreover, if $T$ is a
full column rank matrix with the columns forming a basis for 
$\mathcal{N}(W)$, the null space of $W$, 
then a facial reduction step for the optimal face is the substitution
\[
	Y= T RT^T, \, R \in \Ss_+^{(n+m)-(p+q-2r)}.
\]
\end{lemma}
\begin{proof}
That $\bar U, \bar V$ are exposing vectors is by construction. The
result follows from the fact that the sum of exposing vectors is an
exposing vector. Moreover, the block diagonal structure of the exposing
matrices guarantees that the ranks add up to get the size of $R$. 
(More details are available in \cite{DrPaWo:14,ChDrWo:14}.)
\end{proof}

\subsection{Cliques, Weights and Final Exposing Vector}
Given a partial matrix $Z\in \Rmn$, we need to find nontrivial cliques according
to the definition in \eqref{eq:cardspq} and \eqref{eq:Xspecif}.
Intuitively, we may want to find cliques with size as large as possible so that 
we can expose $Y$ immediately. However, we do
not want to spend a great deal of time finding large cliques.
Instead we find it is more efficient to find many medium-size cliques 
that can cover as many 
vertices as possible. We can then add the exposing vectors obtained 
from these cliques to finally expose a small face containing the optimal $Y$.  
This is equivalent to dealing with a small number of large cliques.
This consideration also comes from the expensive computational cost of 
the eigenvalue calculation for
$\bar U,\bar V$ in \eqref{eq:qpuv} when the clique is large.

The cliques are found through using the \textdef{adjacency matrix} 
defined above in \eqref{eq:Aadj}. We can then use these cliques to 
find a set of exposing vectors. Specifically, we can obtain at most two 
useful exposing vectors from each of the cliques we found.
Exposing vectors are useful only  if they are nonzero. To get a nonzero
exposing vector we need the sizes of the sampling matrix to be
sufficiently large, i.e.,~a
useful exposing vector requires that the diagonal block formed from 
one of the full rank  decomposed parts of this clique has a 
\emph{correct rank} and correct size. 
The correct rank and size of the diagonal block depend on the size and rank 
of the submatrix $X$. In particular, we want at least one of 
$\bar U, \bar V$ in \eqref{eq:qpuv} to be nonzero, in which case, 
we say the clique is useful. We illustrate this in detail in Algorithm 
\ref{alg:findexposvctr}.

The following Lemma shows that, generically, we can restrict the 
search to cliques corresponding to a specified submatrix
$X \in \R^{p\times q}$ such that $\min\{p,q\} \geq r$ without losing rank
magnitude, where $p$ and $q$ are defined in (\ref{eq:cardspq}). This means 
if either $p > r$ or $q > r$, we can obtain a useful exposing vector based on 
that part of the clique.

\begin{lemma}
\label{statlemma}
Let $Z \in \Rmn$ be a random matrix where the entries come from a
continuous distribution.  Suppose that $\rank(Z) = r$ and
$X\in \R^{p \times q}$ is a specified partial matrix obtained from $Z$ with 
$\min\{p,q\} \geq r$. Then $\rank(X) = r$ with probability $1$
(generically).

In terms of the above notation, let
	\[
		Y,Z, r_Y, r_Z, r_{X},p,q,
	\]
be defined as above with $r_Z \leq \min\{p,q\}$. 
Then generically
\[
	\rank \left(	\begin{bmatrix} PDP^T\end{bmatrix}\right)
=\rank \left(	\begin{bmatrix} QDQ^T\end{bmatrix}\right) =
	 r_Z=r_{X}.
\]
\end{lemma}
\begin{proof}
Recall that the $\rank : \Rmn \to \N$ is a lower semi-continuous function.
Therefore, arbitrary small perturbations can increase the rank but not
decrease it. The result now follows since the rank of a submatrix is
bounded above by $r$.

More precisely, without loss of generality, we can
suppose that $X = [x^1, ..., x^r]$, where $x^i \in \Rp$ are
the column vectors of $X$, with $p \geq r$. 
If $\rank(X) < r$, then there exists
$\{a_1, ..., a_r\} \subset \R$ such that $y = \sum_{i=1}^r a_ix^i = 0$.

The first element of this vector is $y_1 = \sum_{i=1}^r a_ix_{1}^i$. 
Since $x_{1}^i$ 
comes from a continuous distribution then so does $y_1$, and 
the probability
$\text{P}(\rank(X) < r) \leq \text{P}(y_1 = 0) = 0$. Thus, it must be that 
$\rank(X) = r$ for $X$ of the appropriate size given in the lemma.
\end{proof}

With the existence of noise, we know that generically 
the $X$ found can only have 
a higher rank but not a lower rank than r. In this case, since we know 
the correct rank of $X$, we can adjust the exposing vector so that it will 
not over-expose the completion matrix.

After finding a clique corresponding to a sampled submatrix and its full
rank factorization $X=\bar P\bar Q^T$, we then construct a sized
\textdef{clique weight, $u^i_X$},
\index{$u_X$, clique weight, $u^i_X, i=P,Q$}
to measure how \emph{noisy} the corresponding exposing vector is. 
We essentially use the \textdef{Eckart-Young distance} to the
nearest matrix of rank $r$ on the semidefinite cone and include
the size. If the problem is
\emph{noiseless}, then generically we expect this distance to be $0$,
since submatrices of sufficient size yield either $\bar P\bar P^T$ or 
$\bar Q\bar Q^T$ to be rank $r$.

\begin{defi}[clique weights]
Let $X=\bar P\bar Q^T$ denote a $p\times q$ sampled submatrix  with its
full rank factorization. If $p>r$ let $B=\bar P\bar P^T=UDU^T$ be the spectral
decomposition with eigenvalues $\lambda_j, j=1,\ldots,p$ in
nondecreasing order. Define the \textdef{clique weight, $u^i_X$}, with
$i=p$
\index{$u^i_X$, clique weight}
$$ 
u^p_X := \frac{\sum_{i = 1}^{p - r} \lambda_i^2 + 
\sum_{i = p - r + 1}^{p}(\min\{0, \lambda_i\})^2}{0.5 p(p - 1)}.
$$
If $q>r$, repeat with $\bar Q$ and $q$.
\end{defi}
\begin{defi}[exposed vector weights]
Define the \textdef{exposed vector weight, $w^i_X$}, as
\index{$w^i_X$, exposed vector weight}
$$w^i_X = 1 - \frac{u^i_X}{\text{sum of all existing clique weights}},
\quad i=p,q.
$$
\end{defi}

Algorithm \ref{alg:findexposvctr} summarizes how to find an exposing
vector $Y_{expo}$ for our optimal $Y$ for the minimum nuclear norm problem.
This exposing vector locates a face containing $Y$ which is
$$F_Y = V\Ss_+^{r_v}V^T$$
where V is from the spectral decomposition of $Y_{expo}$
$$Y_{expo} = \begin{bmatrix} U & V\end{bmatrix}
	\begin{bmatrix} \Sigma & 0 \\ 0 & 0\end{bmatrix}
	\begin{bmatrix} U & V \end{bmatrix}^T,$$
%We then let find $V$ where the columns provide an orthonormal basis for
%$\Null Y_{expo})$, 
such that $\Sigma \succ 0$. In other words,~\FR yields a representation of $Y$ as
\begin{equation}
	\label{eq:VRVt}
	Y=VRV^T, \quad \text{for some} \quad R\in \Ss_+^{r_v},
\end{equation}
where we hope that we have found enough cliques to get the reduction $r_v=r$.
We now aim to find an appropriate $R$.
\begin{algorithm}[ht]%[Exposing vector for $Y$][H]
\caption{Finding Exposing Vectors}
	\begin{algorithmic}[1]
	\label{alg:findexposvctr}
	\STATE \textbf{INPUT:} A partial matrix $Z \in \MM^{m \times n}$, 
		target rank $r$, clique size range $\{minsize, maxsize\}$;
		\STATE \textbf{OUTPUT:} A final \underline{exposing
		vector} that exposes a face containing the matrix 
		$Y \in \Smn$ formed by $Z$
	\STATE \textbf{PREPROCESSING:} \\
		1. form the corresponding adjacency matrix $A$; \\
		2. find a set of cliques $\Theta$ from $A$ of size within 
			the given range;
	\FOR{each clique $X \in \Theta$}
	\STATE 
		$[p,q] \leftarrow \size(X)$; \\
		$[P,Q] \leftarrow \text{FullrankDecompose}(X)$;
		\IF{$p > r$}
		\STATE 
			$W \leftarrow PP^T$; \\
			$[U_X^p,D] \leftarrow \eig(W)$, eigenvalues in
			nondecreasing order; \\
		 calculate clique weight $u^p_{X}$;
		\ENDIF
		\IF{$q > r$}
		\STATE 
			$W \leftarrow QQ^T$; \\
			$[U_X^q,D] \leftarrow \eig(W)$, eigenvalues in
			nondecreasing order\\
		 calculate clique weight $u^q_{X}$;
		\ENDIF
	\ENDFOR
	\STATE  calculate all the exposing vector weights $w^i_{X},
	i=p,q, X\in \Theta$ from \underline{existing} clique weights;\\
	\STATE sum over existing weights using nullity eigenspaces
\[
	\textdef{$Y_{expo}$} \leftarrow \sum_{\stackrel{i=p,q}{w^i_X exists, X \in \Theta}} 
w^i_X\left(U_X^i(:,1:i-r)(U_X^i(:,1:i-r))^T\right);
\]
	\RETURN $Y_{expo}$
	\end{algorithmic}
\end{algorithm}

\section{Dimension Reduction and Refinement}
\label{sect:dimred}
After \FR the original $Y$ can be expressed as $Y = VRV^T$,
where $R \in \Ss^{r_v}$ and $V^TV=I$. This means the problems
\eqref{sdpnuclear} and \eqref{eq:pdpairnuclnormSDP} are
in general reduced to the much smaller dimension $r_v$. And if we find
enough cliques we expect a reduction to $r_v=r$, the target rank.

\subsection{Noiseless Case}
The expression of $Y$ after \FR means we now turn to solve the nuclear 
norm minimization problem
\begin{equation}
\label{frnuclearnoiseless}
\begin{array}{cl}
	\min &  \Trace (R) \quad (=\Trace(VRV^T))  \\
	\text{s.t.}& \PP_{\bar E}(VRV^T) = b\\
	& R \succeq 0,
\end{array}
\end{equation}
where $b=\PP_{\hat E}(Z)$. 
The \FR typically results in many of the linear equality
constraints becoming redundant. We use the compact QR
decomposition\footnote{We use $[\sim,R,E]=qr(\Phi,0)$ to find the list of
constraint for a well conditioned representation, where $\Phi$ denotes the matrix of constraints.}
to identify which constraints to choose that result in a linearly
independent set with a relatively low condition number.
%MATLAB command \emph{licols} that is based on the economy sized
%$QR$-decomposition and produces a \emph{good} set
%of constraints that are linearly independent. 
Thus we have eliminated a portion of the sampling and we get the linear
system
\begin{equation}
	\label{eq:smallsyst}
	\MM(R):=	\PP_{\tilde E}(VRV^T) =  \tilde b, \text{  for some } 
	        \tilde E \subseteq \bar E,
\end{equation}
and $\tilde {b}$ is the vector of corresponding elements in $b$.

\begin{enumerate}
	\item
		\label{item:pd}
We first solve the simple semidefinite constrained least squares problem
\[
	\min_{R \in \Ss_+^{r_v}} \|(V R V^T)_{\tilde E} - \tilde b) \|
\]
If the optimal $R$ has the correct target rank, then the exactness of the data
implies that necessarily the optimal value is zero; and we are done.

\item If $R$ does not have the correct rank in Item \ref{item:pd} above, 
then we solve \eqref{frnuclearnoiseless} for our minimum nuclear norm solution.
%The results are reported in Tables
%\ref{table:noiseless2},\ref{table:noiseless3}.
We note that the linear transformation $\MM$ in \eqref{eq:smallsyst}
may not be one-one.
Therefore, we often need to add a small regularizing term to the
objective, i.e.,~we use
	$\min  \trace(R) + \gamma \|R\|_F$ with small $\gamma>0$.
\end{enumerate}

\subsection{Noisy Case}
\subsubsection{Base Step after Facial Reduction}
As for the noiseless case we complete \FR and expect the dimension
of $R$, $r_v$, to be reduced dramatically.
We again begin and solve the simple semidefinite
constrained least squares problem
\[
	\delta_0 = \min_{R \in \Ss_+^{r_v}} \|(VRV^T)_{\bar E} - b) \|, 
	\quad b=Z_{\hat E}. 
\]

However, unlike in the noiseless case, we cannot remove redundant constraints, even
though there may be many. This problem is now highly overdetermined and
may also be ill-posed in that the constraint transformation may not be
one-one. We  use the notion of \textdef{sketch matrix} to reduce the size of
the system, e.g.,~\cite{PolanciWainwright:15}. The matrix $A$ is a
random matrix of appropriate size with a relatively small number of rows
in order to dramatically decrease the size of the problem. As noted
in~\cite{PolanciWainwright:15}, this leads to surprisingly good results.
If $s$ is the dimension of $R$, then
we use a random sketch matrix of size $2t(s)\times |\hat E|$, where
$t(\cdot)$ is the number of variables on and above the diagonal of a
symmetric matrix, i.e.,~the triangular number
$$t(s) = \frac{s(s+1)}{2}.$$

If the optimal $R$ has the correct target rank, then we are done.

\subsubsection{Refinement Step with Dual Multiplier}
If the result from the base step does not have the correct rank, we 
now use this $\delta_0$ as a best target value for our parametric
approach as done in~\cite{ChDrWo:14}.
Denoting $b=Z_{\hat E}$ as the vector of 
known entries in $Z$ in column order, our minimum nuclear norm
problem can be stated as:
\begin{equation}
\label{frsdpnuclear}
\begin{array}{cl}
	\min & \trace(R) \\
	\text{s.t.}& \|(VRV^T)_{\hat E} - b) \| \leq \delta_0 \\
	& R \succeq 0.
\end{array}
\end{equation}

To ensure a lower rank solution is obtained through this process, we use the approach in~\cite{ChDrWo:14} and \emph{flip}
this problem:
\begin{equation}
	\label{eq:flipNoisyfinalFRSDP}
	\begin{array}{cl}
\varphi(\tau) := \min &  
		    \| \left(\hat V R \hat V^T\right)_{\hat E}  - b \| 
		    	+\gamma \|R\|_F  \\
		\text{s.t.} &  \trace(R) \leq \tau\\
			    & R \succeq 0.
	\end{array}
\end{equation}
As in the noiseless case, the least squares problem may be
underdetermined. We add a
regularizing term $+\gamma \|R\|_F$ to the objective with $\gamma>0$
small. The starting value of $\tau$ is obtained from the unconstrained least 
squares problem, and from which we can shrink the trace of $R$ to reduce the 
resulting rank. We refer to this process as the refinement step. 

This process requires a tradeoff between low-rank and low-error. 
Specifically, the trace constraint may not be tight at the starting value 
of $\tau$, which means we can lower the trace of $R$ without sacrificing accuracy, however, if the trace is pushed lower than necessary, 
the error starts to get larger. To detect the balance point between 
low-rank and low-error, we exploit \textdef{dual multiplier} of the
inquality constraint.
The value of the dual variable indicates the rate of increase
of the objective function. When the 
the dual multiplier becomes positive then we know that decreasing
$\tau$ further will increase the residual value. We have used the value
of $.01$ to indicate that we should stop decreasing $\tau$.

%We decide the value of $\gamma =????$ by ????????????

\section{Numerics}
\label{sect:numers}
We now present experiments with the algorithm on random instances. 
Averages (Times, Rank, Residuals) 
on \underline{\bf five} random instances are included in the table
\footnote{The density $p$ in the tables are reported as ``mean($p$)'' because
the real density obtained is usually not the same as the one set for 
generating the problem. We report the mean of the real densities over the 
five instances.}.
%%
%%We first specify our data generation procedure, and then present a 
%%separate section with the noiseless case.
%%We used instances with relatively large density. These instances
%%do not require any refinement and so no refinement columns are
%%included, see Tables \ref{table:noiseless2},\ref{table:noiseless3}. 
%%Noiseless instances with
%%lower density appear as the first instances in the noisy tables.
%%
In the noisy cases we include the output for both before refinement and
after refinement. (Total time for both is give after refinement.)
We see that in most cases with sufficient density 
refinement is \emph{not} needed.
And, we see that near perfect completion (recovery) is obtained relative to the
noise.  In particular, the low target rank was attained most times.

The tests were run on MATLAB version R2016a, 
on a Dell Optiplex 9020, with Windows 7,
Intel(R) Core(TM) i7-4770 CPU @ 3.40GHz and 16 GB RAM.
For the semidefinite constrained least problems we used the MATLAB addon
CVX \cite{cvxunpubl} for simplicity. This means our cputimes could be
improved if we replaced CVX with a recent \SDP solver.

\subsection{Simulated Data}
We generate the instances as done in the recent work \cite{FangLiuTohZhou:15}.
The target matrices are obtained from
$Z=Z_LZ_R^T$, where $Z_L \in \mathbb R^{m \times r}$ and 
$Z_R \in \mathbb R^{r \times n}$.
Each entry of the two matrices $Z_L$ and $Z_R$ is generated independently 
from a standard normal distribution $N(0,1)$.
For the noisy data, we perturb the known entries by additive noise, i.e., 
\[
	Z_{ij} \leftarrow Z_{ij}+ \sigma \xi_t \|Z\|_\infty, 
	       \quad \forall ij \in \bar E,
\]
where $\xi_t \sim N(0,1)$ and $\sigma$ is a noise factor that can be changed.

We evaluate our results using the same measurement as in 
\cite{FangLiuTohZhou:15}, which we call ``Residual'' in our tables. 
It is calculated as:
	$$\text{Residual} = \frac{\|\hat{Z} - Z\|_F}{\|Z\|_F},$$
where $Z$ is the target matrix, $\hat{Z}$ is the output matrix that we
find, and $\norm{\cdot}_F$ is the Frobenius norm.
\index{Frobenius norm, $\norm{\cdot}_F$}
\index{$\norm{\cdot}_F$, Frobenius norm}

%or?????\cite{MR3440180} deal with ????? we
%compare ??????
%We use random generation to obtain the data matrix Z in the following way ...

We observe that we far outperform the results in \cite{FangLiuTohZhou:15}
both in accuracy and in time; and we solve much larger problems.
We are not as competitive for the low density problems as our method
requires a sufficient number of cliques. We could combine our
preprocessing approach using the cliques before the method
in \cite{FangLiuTohZhou:15} is applied.

%%s dt (r, SR) 0 500 (5, 0.10) 1.1 × 10-3
%%                    (5, 0.15) 7.7 × 10-4
%%                    (10, 0.10) 5.5 × 10-2
%%                    (10, 0.15) 1.3 × 10-3
%%1000 (5, 0.10) 1500 
%%     (5, 0.15) 
%%     (5, 0.20) 
%%     (10, 0.10) 
%%     (10, 0.15) 
%%     (10, 0.20) 
%%0.01
%%RE
%%Max
%%Time
%%Hybrid
%%RE Time RE Time
%%6.9 4.1 × 10-2 12.0 4.0 × 10-2 12.5
%%7.0 3.8 × 10-2 11.1 2.9 × 10-2 13.4
%%8.0 1.2 × 10-1 11.4 2.9 × 10-2 13.4
%%8.6 3.8 × 10-2 11.7 2.3 × 10-2 12.0
%%8.8 × 10-4 44.8 2.9 × 10-2 110.4 1.9 × 10-2 115.1
%%   6.6 × 10-4 43.3 1.9 × 10-2 111.3 1.8 × 10-2 114.3
%%      5.5 × 10-4 44.6 1.8 × 10-2 112.4 6.7 × 10-3 120.0
%%         1.5 × 10-3 44.4 2.9 × 10-2 108.7 2.0 × 10-2 121.7
%%            1.0 × 10-3 45.8 2.0 × 10-2 112.8 1.3 × 10-2 117.8
%%               8.3 × 10-4 45.5 1.5 × 10-2 110.8 8.9 × 10-3 117.3
%%(5, 0.10) 8.1 × 10-4 162.8 2.3 × 10-2 385.4 1.2 × 10-2 408.2
%%(5, 0.15) 6.3 × 10-4 158.3 1.7 × 10-2 396.9 1.1 × 10-2 406.6
%%(5, 0.20) 5.3 × 10-4 158.1 1.3 × 10-2 410.9 5.6 × 10-3 405.3
%%(10, 0.10) 1.3 × 10-3 165.9 2.0 × 10-2 413.8 1.5 × 10-2 413.3
%%(10, 0.15) 9.5 × 10-4 160.8 1.4 × 10-2 410.1 1.3 × 10-2 423.2
%%(10, 0.20) 7.8 × 10-4 161.0 1.2 × 10-2 395.1 7.0 × 10-3 398.2
%%
%%
%%
%%

\subsection{Noiseless Instances}
In Tables \ref{table:noiseless3} and 
\ref{table:noiseless4} we present the results with noiseless
data with  target rank $r=2$ and $r=4$, respectively.
We have left the density of the data relatively high. Note that we
set the density in MATLAB at $.35$ and $.4$ and obtained
$.30$ and $.36$, respectively,
as the average of the actual densities for the $5$ instances.
We see that we get efficient \emph{high} accuracy recovery in \emph{every} 
instance. The accuracy is significantly higher than what one can expect
from an \SDP interior point solver. The cputime is almost entirely spent on
a QR factorization that is used as a heuristic for finding a correct subset of
well-conditioned linear constraints. However, we do not need any
refinement steps as the high density guarantees that we have enough
cliques to cover the nodes in the corresponding graph.
Table \ref{table:sparsetable3} illustrates a different approach for
lower density problems. We remove the rows and columns of the original
data matrix corresponding to zero diagonal elements of the final exposing
matrix. We include the percentage
of the number of elements of the original data matrix that are recovered
and the corresponding percentage residual. Since the accuracy is quite
high for this recovered submatrix, it can then be used as data with another
heuristic, such as the nuclear norm heuristic, to recover the complete matrix.

Note that the largest problems have $50,000,000$ data entries in $Z$ with
approximately $5,000,000$ unkown values that were recovered successfully. 
The correct rank was recovered in every instance.

%In Table \ref{table:noiseless4} we increase the rank and the density and
%still obtain excellent results.
%Note that we consider lower density and higher ranks in the noisy cases in the
%tables that follow below.
	\begin{table*}[!h]
	\centering
	\caption{\underline{noiseless}: $r=2$;
		$m \times n$ size; density $p$}
		\label{table:noiseless3}
		%\begin{tabular}{|c|cccccccc|}
		%\hline
		\input{table_noiseless3}
		%\hline
		%\end{tabular}
\end{table*}
	\begin{table*}[!h]
	\centering
	\caption{\underline{noiseless}: $r=4$;
		$m \times n$ size; density $p$.}
		\label{table:noiseless4}
		%\begin{tabular}{|c|cccccccc|}
		%\hline
		\input{table_noiseless4}
		%\hline
		%\end{tabular}
\end{table*}

	\begin{table*}[!h]
	\centering
	\caption{\underline{sparse data; noiseless}: $r=3$;
		$m \times n$ size; density $p$}
		\label{table:sparsetable3}
		%\begin{tabular}{|c|cccccccc|}
		%\hline
		\input{sparsetable3}
		%\hline
		%\end{tabular}
\end{table*}

\subsection{Noisy Instances}
The first noisy cases follow in Tables
\ref{table:table1} and \ref{table:noisy2}.
%Table \ref{table:noisy3}, page~\pageref{table:noisy3}.
As above for the noiseless case we consider problems with relatively
high density to ensure that we can find enough cliques.
There were some instances where the \emph{eigs} command failed in MATLAB. 
The algorithm avoids these cases by rounding \emph{tiny} numbers to zero and 
finding fewer eigenvalues.

In Table \ref{table:table1}  we consider first increasing noise and then 
increasing size. In Table \ref{table:noisy2} we allow for larger size
and have decreasing density.

	\begin{table*}[!h]
	\centering
	\caption{\underline{noisy}: $r=3$;
	$m \times n$ size; density $p$}
		\label{table:table1}
		%\begin{tabular}{|c|cccccccc|}
		%\hline
		\input{table1}
		%\hline
		%\end{tabular}
\end{table*}
	\begin{table*}[!h]
	\centering
	\caption{\underline{noisy}: $r=3$;
		$m \times n$ size; density $p$}
		\label{table:noisy2}
		%\begin{tabular}{|c|cccccccc|}
		%\hline
		\input{tableNsyLrgrank3}
		%\hline
		%\end{tabular}
\end{table*}

% \begin{table*}[!h]
% 	\centering
% 	\caption{\underline{noisy/mean($5$) instances}: $r=3$;
% 		$m \times n$ size $\uparrow$; 
% 	mean(densities) $p \downarrow$; noise $\uparrow$.}
% 		\label{table:noisy3}
% 		%\begin{tabular}{|c|cccccccc|}
% 		%\hline
% 		\input{tableNsyLrgrank3}
% 		%\hline
% 		%\end{tabular}
% \end{table*}

\subsection{Sparse Noisy Instances}
We now consider our last case - noisy instances but with lower density, see
Tables~\ref{table:sparsenoisytable1},~\ref{table:sparsenoisytable2}.
 
In this case there may not be enough cliques to cover the entire graph for the
problem instance. Moreover, the covered nodes are not covered
\emph{well} and so we do not expect good recovery from this poor data in
the presence of noise.
We report on the percentage of the matrix $Z$ that has been recovered.  
If we wanted to recover more then we
could solve a larger \SDP problem as done in \cite{ChDrWo:14}.

%\begin{adjustbox}{max width=1.2\textwidth}
\begin{table*}[!h]
	\centering
	\caption{\underline{sparse noisy}: $r=2$;
		$m \times n$ size; density $p$}
		\label{table:sparsenoisytable1}
		%\begin{tabular}{|c|cccccccc|}
		%\hline
		\input{sparsenoisytable1}
		%\hline
		%\end{tabular}
\end{table*}
%\end{adjustbox}
%\begin{adjustbox}{max width=1.2\textwidth}
\begin{table*}[!h]
	\centering
	\caption{\underline{sparse noisy}: $r=3$;
		$m \times n$ size; density $p$}
		\label{table:sparsenoisytable2}
		%\begin{tabular}{|c|cccccccc|}
		%\hline
		\input{sparsenoisytable2}
		%\hline
		%\end{tabular}
\end{table*}
%\end{adjustbox}

\section{Conclusion}
\label{sect:concl}
In this paper we have shown that we can apply facial reduction through 
the exposing vector approach used
in \cite{ChDrWo:14} in combination with the nuclear norm
heuristic to efficiently find low-rank matrix completions.
This exploits the degenerate structure of the optimal solution set even
though the nuclear norm heuristic problem itself satisfies strict feasibility.

Specifically, whenever enough cliques are available for our graph
description, we are able to find a proper face with a
\emph{significantly reduced 
dimension} that contains the optimal solution set. We then
solve this smaller minimum trace problem by
\emph{flipping} the problem and using a refinement
with a parametric point approach. If we cannot find enough cliques, 
the matrix can still be partially completed.
Having an insufficient number of cliques is indicative of not having
enough initial data to recover the unknown elements.
Throughout we see that the facial reduction both regularizes the problem
and reduces the size and often allows for a solution without any
refinement.

Our \underline{\emph{preliminary}} numerical results are promising as they
efficiently and accurately recover large scale problems. The numerical
tests are ongoing with improvements in the efficiency of exploiting the
block structure of the cliques and with solving the lower dimensional
flipped problems. In addition, there are many theoretical
questions about the complexity of exact recovery guarantees and the
relation to the number and size of the cliques.

\section*{Acknowledgement}
The authors would like to thank Nathan Krislock for his help with parts
of the MATLAB coding.
\addcontentsline{toc}{section}{Acknowledgement}

\bibliographystyle{plain}

\cleardoublepage
\addcontentsline{toc}{section}{Index}
%\label{ind:index}
\printindex

\addcontentsline{toc}{section}{Bibliography}
%\bibliography{.master,.edm,.psd,.bjorBOOK}
\def\cprime{$'$} \def\cprime{$'$} \def\cprime{$'$}
  \def\udot#1{\ifmmode\oalign{$#1$\crcr\hidewidth.\hidewidth
  }\else\oalign{#1\crcr\hidewidth.\hidewidth}\fi} \def\cprime{$'$}
  \def\cprime{$'$} \def\cprime{$'$}

\end{document}

%% file: table_noiseless3.tex
\begin{tabular}{|ccc||c|c|c|} \hline
\multicolumn{3}{|c||}{Specifications} & \multirow{2}{*}{Time (s)} & \multirow{2}{*}{Rank} &\multirow{2}{*}{Residual (\%$Z$)}\cr\cline{1-3}
  $m$ &   $n$ & mean($p$) &        &         &           \cr\hline
  700 &  2000 &  0.30 &   9.00 &     2.0 & 4.4605e-14 \cr\hline
 1000 &  5000 &  0.30 &  28.76 &     2.0 & 3.0297e-13 \cr\hline
 1400 &  9000 &  0.30 &  77.59 &     2.0 & 7.8674e-14 \cr\hline
 1900 & 14000 &  0.30 & 192.14 &     2.0 & 6.7292e-14 \cr\hline
 2500 & 20000 &  0.30 & 727.99 &     2.0 & 4.2753e-10 \cr\hline
\end{tabular}

%% file: table_noiseless4.tex
\begin{tabular}{|ccc||c|c|c|} \hline
\multicolumn{3}{|c||}{Specifications} & \multirow{2}{*}{Time (s)} & \multirow{2}{*}{Rank} &\multirow{2}{*}{Residual (\%$Z$)}\cr\cline{1-3}
  $m$ &   $n$ & mean($p$) &        &         &           \cr\hline
  700 &  2000 &  0.36 &  12.80 &     4.0 & 1.5217e-12 \cr\hline
 1000 &  5000 &  0.36 &  49.66 &     4.0 & 1.0910e-12 \cr\hline
 1400 &  9000 &  0.36 & 131.53 &     4.0 & 6.0304e-13 \cr\hline
 1900 & 14000 &  0.36 & 291.22 &     4.0 & 3.4847e-11 \cr\hline
 2500 & 20000 &  0.36 & 798.70 &     4.0 & 7.2256e-08 \cr\hline
\end{tabular}

%% file: sparsetable3.tex
\begin{tabular}{|ccc||c|c|c|c|c|} \hline
\multicolumn{3}{|c||}{Specifications} & \multirow{2}{*}{Recover (\%$Z$)} & \multirow{2}{*}{Time (s)} & \multirow{2}{*}{Rank} &\multirow{2}{*}{Residual (\%$Z$)}\cr\cline{1-3}
  $m$ &   $n$ &   $p$ &         &          &         &         \cr\hline
  700 &  1000 &  0.36 &  100.00 &    5.21  &    3.00 & 4.32e-11 \cr\hline
  700 &  1000 &  0.33 &  100.00 &    5.13  &    3.00 & 5.69e-11 \cr\hline
  700 &  1000 &  0.30 &  100.00 &    4.78  &    3.00 & 7.04e-11 \cr\hline
  700 &  1000 &  0.26 &   99.69 &    4.79  &    3.00 & 3.11e-10 \cr\hline
  700 &  1000 &  0.22 &   97.77 &    4.36  &    3.00 & 7.66e-05 \cr\hline
	\hline
 1100 &  8000 &  0.36 &  100.00 &  325.58  &    3.00 & 6.53e-10 \cr\hline
 1100 &  8000 &  0.33 &  100.00 &  321.58  &    3.00 & 3.72e-11 \cr\hline
 1100 &  8000 &  0.30 &  100.00 &  316.28  &    3.00 & 1.92e-10 \cr\hline
 1100 &  8000 &  0.26 &  100.00 &  313.04  &    3.00 & 5.60e-10 \cr\hline
 1100 &  8000 &  0.22 &  100.00 &  307.48  &    3.00 & 9.24e-10 \cr\hline
\end{tabular}

%% file: table1.tex
\begin{tabular}{|cccc||cc|cc|cc|} \hline
\multicolumn{4}{|c||}{Specifications} & \multicolumn{2}{|c|}{Time (s)} & \multicolumn{2}{|c|}{Rank} &\multicolumn{2}{|c|}{Residual (\%$Z$)} \cr\hline
  $m$ &   $n$ & \% noise &   $p$ &initial &  refine & initial &  refine & initial &  refine \cr\hline
  700 &  1000 &  0.00 &  0.36 &   4.50 &    6.35 &    3.00 &    3.00 & 4.14e-14 & 3.99e-14 \cr\hline
  700 &  1000 &  1.00 &  0.36 &   4.22 &    9.00 &    3.00 &    3.00 & 2.52e-02 & 2.52e-02 \cr\hline
  700 &  1000 &  2.00 &  0.36 &   4.27 &    9.13 &    2.60 &    2.60 & 4.05e-01 & 4.04e-01 \cr\hline
  700 &  1000 &  3.00 &  0.36 &   4.15 &    9.43 &    2.20 &    2.20 & 4.31e-01 & 4.30e-01 \cr\hline
  700 &  1000 &  4.00 &  0.36 &   4.30 &   12.34 &    1.60 &    1.60 & 9.64e-01 & 9.61e-01 \cr\hline
	\hline
  700 &  1000 &  1.00 &  0.36 &   4.23 &    8.88 &    3.00 &    3.00 & 2.52e-02 & 2.52e-02 \cr\hline
  800 &  2000 &  1.00 &  0.36 &  11.79 &   20.60 &    3.00 &    3.00 & 1.91e-02 & 1.91e-02 \cr\hline
  900 &  4000 &  1.00 &  0.36 &  43.27 &   65.41 &    3.00 &    3.00 & 1.86e-02 & 1.85e-02 \cr\hline
 1000 &  8000 &  1.00 &  0.36 & 156.81 &  204.76 &    3.00 &    3.00 & 1.46e-02 & 1.46e-02 \cr\hline
 1100 & 16000 &  1.00 &  0.36 & 528.60 &  673.97 &    3.00 &    3.00 & 1.51e-02 & 1.51e-02 \cr\hline
\end{tabular}

%% file: tableNsyLrgrank3.tex
\begin{tabular}{|cccc||cc|cc|cc|} \hline
\multicolumn{4}{|c||}{Specifications} & \multicolumn{2}{|c|}{Time (s)} & \multicolumn{2}{|c|}{Rank} &\multicolumn{2}{|c|}{Residual (\%$Z$)} \cr\hline
  $m$ &   $n$ & \% noise &   $p$ &initial &   total & initial &  refine &   initial &    refine \cr\hline
  700 &  1000 &  0.00 &  0.40 &   2.22 &    1.82 &    2.40 &    2.40 & 3.961e-14 & 3.961e-14 \cr\hline
  700 &  1000 &  0.01 &  0.40 &   4.16 &    8.79 &    3.20 &    3.20 & 9.242e-01 & 9.360e-01 \cr\hline
  700 &  1000 &  0.15 &  0.40 &   3.64 &    6.32 &    2.40 &    2.40 & 9.416e-01 & 9.517e-01 \cr\hline
  700 &  1000 &  0.30 &  0.40 &   3.46 &    7.09 &    8.40 &    8.40 & 9.862e-01 & 9.862e-01 \cr\hline
  700 &  1000 &  0.45 &  0.40 &   3.45 &    4.26 &    3.80 &    3.80 & 9.539e-01 & 9.539e-01 \cr\hline
\hline
 1500 &  2000 & 10.00 &  0.40 &  14.07 &   19.13 &    2.40 &    2.40 & 9.281e-01 & 9.360e-01 \cr\hline
 1600 &  2100 & 10.00 &  0.35 &  13.85 &   18.03 &    2.40 &    2.40 & 9.535e-01 & 9.535e-01 \cr\hline
 1700 &  2200 & 10.00 &  0.30 &  10.48 &   30.81 &   11.00 &   11.00 & 8.000e-01 & 8.000e-01 \cr\hline
 1800 &  2300 & 10.00 &  0.25 &   4.22 &   15.22 &    4.60 &    4.60 & 4.000e-01 & 4.000e-01 \cr\hline
% 1900 &  2400 & 10.00 &  0.20 &   0.00 &    0.00 &    0.00 &    0.00 & 0.000e+00 & 0.000e+00 \cr\hline
\hline
 1900 &  2500 & 10.00 &  0.40 &  21.39 &   29.03 &    2.20 &    2.20 & 9.506e-01 & 9.546e-01 \cr\hline
 2000 &  2600 & 10.00 &  0.35 &  18.58 &   50.70 &   10.20 &   10.20 & 9.894e-01 & 9.894e-01 \cr\hline
 2100 &  2700 & 10.00 &  0.30 &  22.75 &   40.97 &    6.40 &    6.40 & 9.759e-01 & 9.759e-01 \cr\hline
 2200 &  2800 & 10.00 &  0.25 &   6.61 &   26.14 &    5.20 &    5.20 & 4.000e-01 & 4.000e-01 \cr\hline
% 2300 &  2900 & 10.00 &  0.20 &   0.00 &    0.00 &    0.00 &    0.00 & 0.000e+00 & 0.000e+00 \cr\hline
\end{tabular}

%% file: sparsenoisytable1.tex
\begin{tabular}{|cccc||c|cc|cc|cc|} \hline
\multicolumn{4}{|c||}{Specifications} & \multirow{2}{*}{Recover (\%$Z$)} & \multicolumn{2}{|c|}{Time (s)} & \multicolumn{2}{|c|}{Rank} &\multicolumn{2}{|c|}{Residual (\%$Z$)}\cr\cline{1-4}\cline{6-11}
  $m$ &   $n$ & \% noise & mean($p$) &        & initial &  refine & initial &  refine & initial &  refine \cr\hline
  700 &  1000 & 0.0e+00 &  0.18 &  99.89 &    3.18 &    7.91 &    2.00 &    2.00 & 2.32e-12 & 1.29e-12 \cr\hline
  700 &  1000 & 1.0e-01 &  0.18 &  99.89 &    3.10 &   13.87 &    2.40 &    2.00 & 1.53e+01 & 1.11e+00 \cr\hline
  700 &  1000 & 2.0e-01 &  0.18 &  99.89 &    2.98 &   14.15 &    2.40 &    2.00 & 2.79e+01 & 2.22e+00 \cr\hline
  700 &  1000 & 3.0e-01 &  0.18 &  99.89 &    3.00 &   14.36 &    2.40 &    2.00 & 3.83e+01 & 3.12e+00 \cr\hline
  700 &  1000 & 4.0e-01 &  0.18 &  99.89 &    2.98 &   15.53 &    2.40 &    2.00 & 4.76e+01 & 4.23e+00 \cr\hline
\hline
  700 &  1000 & 1.0e-03 &  0.33 & 100.00 &    3.64 &   13.08 &    2.80 &    2.00 & 4.11e-03 & 4.11e-03 \cr\hline
  700 &  1000 & 1.0e-03 &  0.30 & 100.00 &    3.27 &   11.71 &    2.50 &    2.00 & 4.49e-03 & 4.49e-03 \cr\hline
  700 &  1000 & 1.0e-03 &  0.26 & 100.00 &    3.78 &   15.18 &    2.20 &    2.00 & 3.52e-03 & 3.52e-03 \cr\hline
  700 &  1000 & 1.0e-03 &  0.22 & 100.00 &    3.88 &   16.54 &    2.00 &    2.00 & 6.62e-03 & 6.62e-03 \cr\hline
  700 &  1000 & 1.0e-03 &  0.18 &  99.89 &    3.66 &   10.34 &    2.40 &    2.00 & 1.35e-01 & 1.35e-01 \cr\hline
\hline
  900 &  2000 & 1.0e-04 &  0.18 & 100.00 &    8.74 &   25.18 &    2.60 &    2.00 & 1.09e-02 & 1.09e-02 \cr\hline
  900 &  2000 & 1.0e-04 &  0.16 & 100.00 &    8.39 &   26.91 &    2.00 &    2.00 & 9.66e-04 & 9.66e-04 \cr\hline
  900 &  2000 & 1.0e-04 &  0.14 &  99.96 &    7.94 &   25.85 &    2.60 &    2.00 & 2.84e-02 & 2.84e-02 \cr\hline
  900 &  2000 & 1.0e-04 &  0.11 &  98.89 &    7.86 &   24.00 &    2.40 &    2.00 & 1.17e-01 & 1.16e-01 \cr\hline
  900 &  2000 & 1.0e-04 &  0.09 &  92.26 &    6.48 &   26.28 &    2.80 &    2.00 & 8.51e-01 & 3.47e-01 \cr\hline
\end{tabular}

%% file: sparsenoisytable2.tex
\begin{tabular}{|cccc||c|cc|cc|cc|} \hline
\multicolumn{4}{|c||}{Specifications} & \multirow{2}{*}{Recover (\%$Z$)} & \multicolumn{2}{|c|}{Time (s)} & \multicolumn{2}{|c|}{Rank} &\multicolumn{2}{|c|}{Residual (\%$Z$)}\cr\cline{1-4}\cline{6-11}
  $m$ &   $n$ & \% noise & mean($p$) &        & initial &  refine & initial &  refine & initial &  refine \cr\hline
  700 &  1000 & 0.0e+00 &  0.18 &  86.90 &    1.87 &    4.85 &    3.20 &    3.00 & 3.76e-07 & 3.66e-07 \cr\hline
  700 &  1000 & 1.0e-01 &  0.18 &  86.75 &    2.81 &   31.12 &    4.25 &    3.75 & 2.07e+03 & 9.78e+01 \cr\hline
  700 &  1000 & 2.0e-01 &  0.18 &  86.75 &    2.77 &   31.60 &    4.25 &    3.50 & 2.88e+03 & 2.70e+02 \cr\hline
  700 &  1000 & 3.0e-01 &  0.18 &  86.75 &    2.69 &   32.35 &    4.00 &    3.25 & 2.96e+03 & 1.59e+02 \cr\hline
  700 &  1000 & 4.0e-01 &  0.18 &  86.90 &    2.61 &   35.53 &    4.00 &    3.60 & 6.09e+04 & 4.10e+02 \cr\hline
  700 &  1000 & 1.0e-03 &  0.33 & 100.00 &    4.93 &   13.04 &    3.00 &    3.00 & 2.43e-03 & 2.43e-03 \cr\hline
  700 &  1000 & 1.0e-03 &  0.30 & 100.00 &    4.62 &   14.36 &    3.80 &    3.00 & 1.77e-02 & 1.77e-02 \cr\hline
  700 &  1000 & 1.0e-03 &  0.26 &  99.69 &    4.19 &   16.04 &    3.00 &    3.00 & 6.94e-02 & 6.94e-02 \cr\hline
  700 &  1000 & 1.0e-03 &  0.22 &  97.77 &    3.81 &   13.91 &    3.40 &    3.00 & 9.74e-01 & 8.42e-01 \cr\hline
  700 &  1000 & 1.0e-03 &  0.18 &  86.75 &    2.93 &   13.23 &    4.75 &    3.00 & 3.54e+00 & 1.65e+00 \cr\hline
  900 &  2000 & 1.0e-04 &  0.18 &  96.81 &    8.01 &   26.45 &    4.60 &    3.00 & 9.72e-02 & 9.71e-02 \cr\hline
  900 &  2000 & 1.0e-04 &  0.16 &  92.60 &    6.04 &   18.93 &    4.80 &    3.00 & 2.10e+00 & 7.20e-01 \cr\hline
  900 &  2000 & 1.0e-04 &  0.16 &  89.45 &    5.41 &   31.18 &    4.00 &    3.25 & 5.72e+01 & 7.57e-01 \cr\hline
  900 &  2000 & 1.0e-04 &  0.15 &  83.53 &    4.71 &   18.10 &    5.00 &    3.00 & 2.61e-01 & 2.60e-01 \cr\hline
  900 &  2000 & 1.0e-04 &  0.14 &  74.94 &    7.21 &   28.49 &    4.00 &    3.00 & 4.05e+01 & 5.70e+00 \cr\hline
\end{tabular}